\documentclass[12pt]{article}

\usepackage{mathrsfs}
\usepackage{color}
\usepackage{amssymb, amsthm, amsfonts, amsxtra, amsmath}

\usepackage{latexsym}

\newtheorem{Theorem}{Theorem}[section]
\newtheorem{Def}[Theorem]{Definition}
\newtheorem{Lem}[Theorem]{Lemma}

\newtheorem{Def-Prop}[Theorem]{Definition-Proposition}
\newtheorem{Not}[Theorem]{Notation}
\newtheorem*{Lemm}{Lemma}

\date{}

\begin{document}

\title{A note on the von Neumann algebra underlying some universal compact quantum groups}
\author{Kenny De Commer\footnote{Research Assistant of the Research Foundation - Flanders (FWO -
Vlaanderen).}\\ \small Department of Mathematics,  K.U. Leuven\\
\small Celestijnenlaan 200B, 3001 Leuven, Belgium\\ \\ \small e-mail: kenny.decommer@wis.kuleuven.be}
\maketitle

\newcommand{\acnabla}{\nabla\!\!\!{^\shortmid}}
\newcommand{\undersetmin}[2]{{#1}\underset{\textrm{min}}{\otimes}{#2}}
\newcommand{\otimesud}[2]{\overset{#2}{\underset{#1}{\otimes}}}

\newcommand{\otimesmin}{\underset{\textrm{min}}{\otimes}}
\newcommand{\bigback}{\!\!\!\!\!\!\!\!\!\!\!\!\!\!\!\!\!\!\!\!\!\!\!\!}

\abstract{\noindent We show that for $F\in GL(2,\mathbb{C})$, the von Neumann algebra associated to the universal quantum group $A_u(F)$ is a free Araki-Woods factor.}\vspace{0.7cm}



\section*{Introduction}

\noindent It is a classical theorem that any compact Lie group is a closed subgroup of some $U(n)$. In \cite{VDae1}, a class of quantum groups was introduced which plays the same r\^{o}le with respect to the compact matrix quantum groups (introduced in \cite{Wor1}, but there called compact quantum \emph{pseudo}groups). These universal quantum groups were denoted $A_u(F)$, where the parameter $F$ takes values in invertible matrices over $\mathbb{C}$. In \cite{Ban1}, the representation theory of the $A_u(F)$ was investigated, and it was shown that the irreducible representations are naturally labeled by the free monoid with two generators. Also on the level of the `function algebra' of $A_u(F)$, freeness manifests itself: it was shown in \cite{Ban1} that the (normalized) trace of the fundamental representation is a circular element w.r.t.~the Haar state (in the sense of Voiculescu, see \cite{Voi1}). Furthermore, the von Neumann algebra associated to $A_u(I_2)$, where $I_2$ is the unit matrix in $GL(2,\mathbb{C})$, is actually isomorphic to the free group factor $\mathscr{L}(\mathbb{F}_2)$.\\


\noindent In this note, we generalize this last result by showing that for $0<q\leq 1$, the von Neumann algebra underlying the universal quantum group $A_u(F)$ with $F=\left(\begin{array}{ll} 1 &0 \\ 0& q\end{array}\right)$ is a free Araki-Woods factor (\cite{Shl1}), namely the one associated to the orthogonal representation \[t\rightarrow \left(\begin{array}{ll} \cos(t\ln q^2) & -\sin(t\ln q^2) \\ \sin(t\ln q^2) & \cos(t\ln q^2)\end{array}\right)\] of $\mathbb{R}$ on $\mathbb{R}^2$. The proof of this fact uses a technique similar to the one of Banica for the case $F=I_2$, combined with results from \cite{Hou1} (which are based on the matrix model techniques from \cite{Shl1}). Since \[A_u(F) = A_u(\lambda U|F|U^*)\] for any $\lambda \in \mathbb{R}^+_0$ and any unitary $U$ (see \cite{Ban1}), we obtain that all $A_u(F)$ with $F\in GL(2,\mathbb{C})$ have free Araki-Woods factors as their associated von Neumann algebras.  \\

\noindent \textit{Remarks on notation:} We denote by $\odot$ the algebraic tensor product of vector spaces over $\mathbb{C}$, and by $\otimes$ the spatial tensor product between von Neumann algebras or Hilbert spaces. If $M$ is a von Neumann algebra and $x_1,x_2,\ldots$ are elements in $M$, we denote by $W^*(x_1,x_2,\ldots)$ the von Neumann subalgebra of $M$ which is the $\sigma$-weak closure of the unital $^*$-algebra generated by the $x_i$.\\



\section{Preliminaries}

\noindent In this preliminary section, we will give, for the sake of economy, ad hoc definitions of the von Neumann algebras associated to the $A_u(F)$ and $A_o(F)$ quantum groups (\cite{VDae1}), and of the free Araki-Woods factors (\cite{Shl1}), for special values of their parameters. \\

\noindent \emph{Throughout this section, we fix a number $0<q<1$.}\\

\begin{Def} We define the C$^*$-algebra $C_u(H)$ as the universal enveloping C$^*$-algebra of the unital $^*$-algebra generated by elements $a$ and $b$, with defining relations \[\left\{\begin{array}{lllclll} a^*a +b^*b = 1 && \!\!\!\!\!\!\!\!\!\!\!\!ab = qba \\aa^* + q^2 bb^* = 1 &&\!\!\!\!\!\!\!\!\!\!\!\! a^*b = q^{-1}ba^* \\ &\!\!\!\!\!\!\!\!bb^* = b^*b.\end{array}\right.\]\end{Def}

\noindent \emph{Remark:} $C_u(H)$ is the (universal) C$^*$-algebra associated with the quantum group $H=SU_q(2)$. In \cite{Ban1}, Proposition 5, it is shown that this equals the quantum group $A_o(\left(\begin{array}{ll}0 & 1 \\ -q^{-1} & 0\end{array}\right))$.\\

\noindent The following fact is found in \cite{Wor2}.

\begin{Lem} Let $\mathscr{H}$ be the Hilbert space $l^2(\mathbb{N}) \otimes l^2(\mathbb{Z})$, whose canonical basis elements we denote as $\xi_{n,k}$ (and with the convention $\xi_{n,k} =0$ when $n<0$). Then there exists a faithful unital $^*$-representation of $C_u(H)$ on $\mathscr{H}$, determined by \[\left\{\begin{array}{l} \pi(a)\, \xi_{n,k} = \sqrt{1-q^{2n}} \xi_{n-1,k},\\ \pi(b)\xi_{n,k} = q^{n} \,\xi_{n,k+1}.\end{array}\right.\]  \end{Lem}

\begin{Def} In the notation of the previous lemma, denote by $\psi$ the state \[\psi(x) = (1-q^2) \sum_{n\in \mathbb{N}} q^{2n} \langle \pi(x) \xi_{n,0},\xi_{n,0}\rangle\] on $C_u(H)$. Then $\psi$ is called the \emph{Haar state} on $C_u(H)$.\end{Def}

\noindent Of course, this name is motivated by the further compact quantum group structure on $C_u(H)$, which we will however not need in the following.\\

\begin{Def} The von Neumann algebra $\mathscr{L}^{\infty}(H)$ is defined to be the $\sigma$-weak closure of $C_u(H)$ in its GNS-representation with respect to the Haar state $\psi$.\end{Def}

\noindent We then continue to write $\psi$ for the extension of $\psi$ to a normal state on $\mathscr{L}^{\infty}(H)$.\\

\begin{Not}\label{Not1}
\noindent We will further use the following notations:
\begin{itemize} \item The matrix units of $B(l^2(\mathbb{N}))$ w.r.t.~the canonical basis of $l^2(\mathbb{N})$ are written $e_{ij}$. \item We denote $\omega$ for the normal state $\omega(e_{ij}) = \delta_{i,j} (1-q^2)q^{2i}$ on $B(l^2(\mathbb{N}))$.
\item We denote by $S \subseteq \mathscr{L}(\mathbb{Z})$ the shift operator $\xi_k\rightarrow \xi_{k+1}$ on $l^2(\mathbb{Z})$.
 \item We denote by $\tau$ the state on $\mathscr{L}(\mathbb{Z})$ which makes $S$ into a Haar unitary with respect to $\tau$.\end{itemize}\end{Not}

\noindent This last fact simply means that $\tau(S^n)=0$ for $n\in \mathbb{Z}_0$.\\

\noindent We will use the terminology `W$^*$-probability space' when talking about a von Neumann algebra with some fixed normal state on it. An isomorphism between two W$^*$-probability spaces is then a $^*$-isomorphism between the underlying von Neumann algebras, preserving the associated fixed states.\\

\begin{Lem} There is a natural isomorphism \[(\mathscr{L}^{\infty}(H),\psi)\rightarrow (B(l^{2}(\mathbb{N}))\otimes \mathscr{L}(\mathbb{Z}),\omega\otimes \tau)\] of W$^*$-probability spaces.\end{Lem}

\begin{proof} By the construction of $\psi$, we may identify $\mathscr{L}^{\infty}(H)$ with $\pi(C_u(H))''$, and it is then sufficient to prove that this last von Neumann algebra equals $B(l^{2}(\mathbb{N}))\otimes \mathscr{L}(\mathbb{Z})$. Clearly, $\pi(C_u(H))''\subseteq B(l^{2}(\mathbb{N}))\otimes \mathscr{L}(\mathbb{Z})$. By functional calculus on $a$ and $b$, we have $e_{ij}\otimes S^n \in \pi(C_u(H))''$ for all $i,j\in \mathbb{N}$ and $n\in \mathbb{Z}$, so in fact equality holds.
\end{proof}

\noindent We will always write $(1\otimes S)$ for the copy of $S \in \mathscr{L}(\mathbb{Z})$ inside $\mathscr{L}^{\infty}(H)$. Hence there should be no notational confusion in the following definition.

\begin{Def}\label{DefTeo} The W$^*$-probability space $(\mathscr{L}^{\infty}(G),\varphi)$ is defined as \[(W^*(Sa,Sb,Sa^*,Sb^*),(\tau*\psi)_{|\mathscr{L}^{\infty}(G)})\subseteq (\mathscr{L}(\mathbb{Z}),\tau)*(\mathscr{L}^{\infty}(H),\psi).\]
\end{Def}

\noindent \emph{Remark:} By \cite{Ban1}, Th\'{e}or\`{e}me 1.(iv), the von Neumann algebra $\mathscr{L}^{\infty}(G)$ will coincide with the von Neumann algebra associated with the universal quantum group $A_u(\left(\begin{array}{ll} 1 &0 \\ 0& q\end{array}\right))$, and $\varphi$ with its Haar state.\\


\noindent Recall that the state $\omega$ was introduced in Notation \ref{Not1}.

\begin{Def}\label{DefFreeAr}(\cite{Shl1}, Corollary 4.9) By a \emph{free Araki-Woods factor (at parameter $q^2$)}, we mean a W$^*$-probability space $(N,\phi)$ isomorphic to the free product $(\mathscr{L}(\mathbb{Z}),\tau)*(B(l^2(\mathbb{N})),\omega)$.
\end{Def}


\section{$\mathscr{L}^{\infty}(G)$ is free Araki-Woods}

\noindent \emph{Throughout this section, we again fix a number $0<q<1$. We also continue to use the notations introduced in the previous section.}\\

\noindent We proceed to prove the following theorem. \begin{Theorem}\label{TheoIso} The W$^*$-probability space $(\mathscr{L}^{\infty}(G),\varphi)$ is a free Araki-Woods factor at parameter $q^2$.\end{Theorem}

\noindent By the remark after Definition \ref{DefTeo} and the remarks in the introduction, this will imply that if $F\in GL(2,\mathbb{C})$, then the von Neumann algebra associated to $A_u(F)$ is the free Araki-Woods factor at parameter $\frac{\lambda_1}{\lambda_2}$, where $\lambda_1\leq \lambda_2$ are the eigenvalues of $F^*F$ (where we take $\mathscr{L}(\mathbb{F}_2)$ to be the free Araki-Woods factor at parameter $1$).\\

\noindent The proof of Theorem \ref{TheoIso} will be preceded by three lemmas. Consider the following von Neumann subalgebras of $(\mathscr{L}(\mathbb{Z}),\tau)*(\mathscr{L}^{\infty}(H),\psi)$: \[(M_1,\varphi_1) = (W^*(S(1\otimes S)), (\tau*\psi)_{\mid M_1})\] and \[(M_2,\varphi_2) = (W^*((1\otimes S^*)a,(1\otimes S^*)b,(1\otimes S^*)a^*,(1\otimes S^*)b^*),(\tau*\psi)_{\mid M_2}).\]\\

\begin{Lem}\label{Lem1} The von Neumann algebras $M_1$ and $M_2$ are free with respect to each other, and $\mathscr{L}^{\infty}(G)$ is the smallest von Neumann subalgebra of $\mathscr{L}(\mathbb{Z})*\mathscr{L}^{\infty}(H)$ which contains them. \end{Lem}

\begin{proof} The proof is entirely similar to the one of Th\'{e}or\`{e}me 6 in \cite{Ban1}. First of all, remark that $S(1\otimes S)$ is the unitary part in the polar decomposition of $Sb$, so that $S(1\otimes S)$ is in $\mathscr{L}^{\infty}(G)$. Then of course \[(1\otimes S^*)a = (1\otimes S^*)S^*\cdot Sa\] is in $\mathscr{L}^{\infty}(G)$, and similarly for the other generators of $M_2$. Hence $M_1$ and $M_2$ indeed generate $\mathscr{L}^{\infty}(G)$.\\

\noindent The proof of the freeness of $M_1$ w.r.t.~$M_2$ is based on a small alteration of Lemme 8 of \cite{Ban1}.

\begin{Lemm}\label{LemBan1} Let $(A,\phi)$ be a unital $^*$-algebra together with a functional $\phi$ on it. Let $B\subseteq A$ be a unital sub-$^*$-algebra, and $d\in B$ a unitary in the center of $B$ such that $\phi(d)=\phi(d^*)=0$. Let $u\in A$ be a Haar unitary which is $^*$-free from $B$ w.r.t.~$\phi$. Then $ud$ is a Haar unitary which is $^*$-free from $B$ w.r.t.~$\phi$. \end{Lemm}

\begin{proof} This is precisely Lemme 8 of \cite{Ban1}, with the condition `$\phi$ is a trace' replaced by `$d$ is in the center of $B$'. However, the proof of that lemma still applies ad verbam.\end{proof}

\noindent We can then apply this lemma to get that $S(1\otimes S)$ is $^*$-free w.r.t.~$\mathscr{L}^{\infty}(H)$, by taking $(A,\phi)=(\mathscr{L}(\mathbb{Z}),\tau)*(\mathscr{L}^{\infty}(H),\psi)$, $B=\mathscr{L}^{\infty}(H)$, $d=1\otimes S$ and $u=S$. \emph{A fortiori}, we will then have $M_1$ free w.r.t.~$M_2$.\end{proof}

\begin{Lem}\label{Lem2} We have \[(M_1, \varphi_1) \cong (\mathscr{L}(\mathbb{Z}),\tau)\] and \[(M_2,\varphi_2)\cong (B(l^{2}(\mathbb{N}))\otimes \mathscr{L}(\mathbb{Z}),\omega\otimes \tau).\]
\end{Lem}

\begin{proof} The fact that $(M_1, \varphi_1) \cong (\mathscr{L}(\mathbb{Z}),\tau)$ is of course trivial. We want to show that $(M_2,\varphi_2)\cong (B(l^{2}(\mathbb{N}))\otimes \mathscr{L}(\mathbb{Z}),\omega\otimes \tau)$.\\

\noindent We have that $1\otimes S^2$ is in $M_2$, since this is the adjoint of the unitary part of the polar decomposition of $(1\otimes S^*)b^*$. Also all $e_{ii}\otimes 1$ are in $M_2$, by functional calculus on the positive part of this polar decomposition. Hence, by multiplying $(1\otimes S^*)a$ or $(1\otimes S^*)a^*$ to the left with the $e_{ii}\otimes 1$, and possibly multiplying with $1\otimes S^2$, we conclude that the $e_{ij}\otimes S^{i-j}$ with $|i-j|=1$ are in $M_2$. But then also all $f_{ij}=e_{ij}\otimes S^{i-j}$ with $i,j\in \mathbb{N}$ are in $M_2$, and it is not hard to see that in fact $M_2=W^*(f_{ij},(1\otimes S^2))$. Since $\psi(f_{ij}(1\otimes S^{2})^n) = (\omega\otimes \tau)(e_{ij}\otimes S^{n})$ by an easy calculation, we are done.\end{proof}

\begin{Lem} \label{Lem3} The W$^*$-probability space $(N,\phi):=(\mathscr{L}(\mathbb{Z}),\tau)*(B(l^{2}(\mathbb{N}))\otimes \mathscr{L}(\mathbb{Z}),\omega\otimes \tau)$ is a free Araki-Woods factor at parameter $q^2$.\end{Lem}

\begin{proof} The proof is completely similar to the one of Theorem 3.1 of \cite{Hou1}. Denote $(N,\theta) = (\mathscr{L}(\mathbb{Z}),\tau)*(B(l^{2}(\mathbb{N})),\omega)$, and denote $\phi_{0} = \frac{1}{1-q^2}\phi$ and $\theta_0=\frac{1}{1-q^2}\theta$. Then by Proposition 3.10 of \cite{Hou1}, we will have that \[(e_{00}Me_{00},\phi_0) \cong (\mathscr{L}(\mathbb{Z}),\tau)*(e_{00}Ne_{00},\theta_0).\] By Proposition 2.7 in \cite{Hou1} (which is based on the proof of Theorem 5.4 and Proposition 6.3 in \cite{Shl1}) and the remark before it, we know that $(e_{00}Ne_{00},\theta_0)$ as well as $(N,\theta) \cong (e_{00}Ne_{00},\theta_0)\otimes (B(l^2(\mathbb{N})),\omega)$ are free Araki-Woods factors at parameter $q^2$. By the free absorption property (\cite{Shl1}, Corollary 5.5), $(e_{00}Me_{00},\phi_0)$ is a free Araki-Woods factor at parameter $q^2$, and hence also $(M,\phi) \cong (e_{00}Me_{00},\phi_0)\otimes (B(l^2(\mathbb{N})),\omega)$ is.

\end{proof}

\begin{proof}[Proof (of Theorem \ref{TheoIso})]  By the first two lemmas, $(\mathscr{L}^{\infty}(G),\varphi)$ is isomorphic to the free product of $(\mathscr{L}(\mathbb{Z}),\tau)$ with $(B(l^{2}(\mathbb{N}))\otimes \mathscr{L}(\mathbb{Z}),\omega\otimes \tau)$, which by the third lemma is a free Araki-Woods factor at parameter $q^2$.
\end{proof}

\vspace{0.3cm}

\noindent \textbf{Acknowledgement:} The motivation for this paper comes from a question posed by Stefaan Vaes concerning the validity of Theorem \ref{TheoIso}.

\end{document}